\documentclass[10pt]{article}
\usepackage{amsmath}
\usepackage{amsthm}
\usepackage{mathrsfs}	
\usepackage{amssymb}
\usepackage{amscd}
\usepackage{setspace}	
\usepackage{tikz}		
\usepackage{tikz-cd}
\usetikzlibrary{matrix}	
\usepackage{graphicx}	
\graphicspath{ {pics/} }	
\usepackage{wrapfig}	
\usepackage{enumerate}	
\usepackage{bbm}		

\usepackage[utf8]{inputenc}	
\usepackage[english]{babel}	
\usepackage{blindtext}		
\usepackage[affil-it]{authblk}	

\usepackage{hyperref}

\usepackage{faktor}			
\usepackage{xparse}			
\DeclareDocumentCommand{\rfaktor}{m O{-0.5} m O{0.5}}{
  	\raisebox{#2\height}{\ensuremath{#1}}
 	 \mkern-5mu\diagdown\mkern-4mu
 	 \raisebox{#4\height}{\ensuremath{#3}}
	}

\newcommand{\norm}[1]{\left\lVert#1\right\rVert}					
\newcommand{\lrangle}[1]{\left\langle#1\right\rangle}					
\newcommand{\cu}[1]{C_u^*\left(#1\right)}							
\newcommand{\cualg}[2]{\mathbb{C}_u^{#1}\left[#2\right]} 			

\DeclareMathOperator{\dd}{d}

\newcommand{\N}{\mathbb{N}}			

\newcommand{\C}{\mathbb{C}}

\newcommand{\D}{\mathbb{D}}

\newcommand{\BB}{\mathscr{B}}

\newcommand{\KK}{\mathscr{K}} 

\newcommand{\HH}{\mathcal{H}} 		

\newtheorem{theorem}{Theorem}[section]     

\newtheorem{lemma}[theorem]{Lemma}

\theoremstyle{remark}

\theoremstyle{definition}
\newtheorem{definition}[theorem]{Definition}

\title{Bounded Derivations on Uniform Roe Algebras}
\author{Matthew Lorentz and Rufus Willett}
\date{}

\begin{document}

\maketitle

\begin{abstract}
We show that if $C^*_u(X)$ is a uniform Roe algebra associated to a bounded geometry metric space $X$, then all bounded derivations on $C^*_u(X)$ are inner.

\end{abstract}

\section{Introduction}
Let $A$ be a $C^*$-algebra.  A \emph{derivation} of $A$ is a linear map $\delta:A\to A$ satisfying $\delta(ab)=a\delta(b)+\delta(a)b$.   In this paper, we always assume that our derivations are defined on all of $A$, and are thus bounded by a fundamental result of Sakai \cite{Sakai:1960aa}.   A derivation $\delta$ of $A$ is \emph{inner} if there exists $d$ in the multiplier algebra $M(A)$ of $A$ such that $\delta(a)=ad-da$ for all $a\in A$.  Let us say that a $C^*$-algebra $A$ \emph{only has inner derivations} if all (bounded) derivations are inner.

Motivated by the needs of mathematical physics and the study of one-parameter automorphism groups, it is interesting to study whether all derivations are inner for a particular $C^*$-algebra.  In the 1970s, a complete solution to this problem was obtained in the separable case via the work of several authors.  The definitive result was obtained by Akemann and Pedersen \cite{Akemann:1979aa}: they showed that a separable $C^*$-algebra has only inner derivations if and only if it isomorphic to a $C^*$-algebra of the form
\begin{equation}\label{nice}
C\oplus \bigoplus_{i\in I} S_i ,
\end{equation}
where $C$ is continuous trace (possibly zero), and each $S_i$ is simple (possibly zero).  Thus in particular, all separable commutative, and all separable simple, $C^*$-algebras only have inner derivations.  However, one might reasonably say that most separable $C^*$-algebras admit non-inner derivations.

For non-separable $C^*$-algebras the picture is murkier.  It is well-known that there are non-separable $C^*$-algebras that are not of the form in line \eqref{nice} and that only have inner derivations: perhaps most famously, Sakai \cite{Sakai:1966aa} has shown this for all von Neumann algebras.  See also for example \cite[page 123]{Elliott:1977aa} for some examples that are not von Neumann algebras, nor of the form in line \eqref{nice}, and that only have inner derivations.  

Our goal in this paper is to give a new class of examples that only have inner derivations: uniform Roe algebras.  Uniform Roe algebras are a well-studied class of non-separable $C^*$-algebras associated to metric spaces; see Section \ref{def sec} below for basic definitions.  They were originally introduced for index-theoretic purposes, but are now studied for their own sake as a bridge between $C^*$-algebra theory and coarse geometry, as well as having interesting applications to single operator theory and mathematical physics, amongst other things.  Due to the presence of $\ell^\infty(X)$ as a diagonal\footnote{In the sense of Kumjian: see \cite{Kumjian:1986aa}} MASA. they have a somewhat von Neumann algebraic flavor, but are von Neumann algebras only in the trivial finite-dimensional case.  They are also essentially never of the form in line \eqref{nice}.  Moreover, in many ways they are quite tractable as $C^*$-algebras, often having good regularity properties such as nuclearity.
 
Here is our main theorem.

\begin{theorem}\label{main}
Uniform Roe algebras associated to bounded geometry metric spaces only have inner derivations.
\end{theorem}

The key ingredients in the proof are: a basic form of a  `reduction of cocycles' argument used by Sinclair and Smith \cite{Sinclair:2004aa} in their study of Hochschild cohomology of von-Neumann algebras; and recent applications of Ramsey-theoretic ideas to the study of uniform Roe algebras by Braga and Farah \cite{Braga:2018dz}.

We conclude this introduction by noting that the fact that all derivations on $A$ are inner can be restated as saying that the first Hochschild cohomology group $H^1(A,A)$ vanishes.  For $A$ a (nuclear) uniform Roe algebra, it is then natural to ask if all the higher groups $H^n(A,A)$ vanish.  See \cite{Sinclair:2004aa} for a survey of this problem in the case that $A$ is a von Neumann algebra.

\subsection*{Acknowledgements} 

The authors would like to thank Roger Smith and Stuart White for useful suggestions.  In particular, an earlier version of this paper only proved the main result in the special case that $C^*_u(X)$ is nuclear, and via a more complicated method.  We are grateful to Roger Smith for suggesting we think about the `reduction of cocycle method', which allowed us to both generalize the main result, and simplify its proof.  Both authors were partially supported by NSF grants DMS-1564281 and DMS-1901522.

\section{Definitions and background results}\label{def sec}

In this section, we recall some basic definitions, as well as a classical result of Kadison stating that all derivations on a $C^*$-algebra are spatially implemented.

Inner products are linear in the first variable. For a Hilbert space $\HH$ we denote the space of bounded operators on $\HH$ by $\BB(\HH)$, and the space of compact operators by $\KK(\HH)$.  The commutator of $a,b\in \BB(\HH)$ is denoted by $[a,b]:=ab-ba$.  

The Hilbert space of square-summable sequences on a set $X$ is denoted $\ell^2(X)$, and the canonical basis of $\ell^2(X)$ will be denoted $(\vartheta_x)_{x\in X}$ (we reserve $\delta$ for derivations).  For $a\in \BB(\ell^2(X))$ we define its matrix entries by
$$
a_{xy}:= \lrangle{a\vartheta_y,\vartheta_x}.
$$

\begin{definition}[propagation, uniform Roe algebra]
Let $X$ be a metric space and $r\geq 0$.  An operator $a\in \BB(\ell^2(X))$ \emph{has propagation at most $r$} if $a_{xy}=0$ whenever 
$d(x,y) > r$ for all $(x,y)\in X\times X$.  In this case, we write prop$(a) \leq r$. 
The set of all operators with propagation at most $r$ is denoted $\cualg{r}{X}$.  We define
$$
\cualg{}{X} := \{a\in \BB(\ell^2(X)) : \text{prop}(a) < \infty\};
$$
it is not difficult to see that this is a $*$-algebra.  The \textit{uniform Roe algebra}, denoted $C^*_u(X)$, is defined to be the norm closure of $\C_u[X]$.
\end{definition}

\begin{definition}[$\epsilon$-$r$-\textit{approximated}]
Let $X$ be a metric space.
Given $\epsilon >0$ and $r >0$, 
an operator $a\in \BB(\ell^2(X))$ can be $\epsilon$-$r$-\textit{approximated} if there exists an 
$b \in \cualg{r}{X}$ such that $\norm{a-b}<\epsilon$.
\end{definition}
	
We will exclusively be interested in uniform Roe algebras associated to bounded geometry metric spaces as in the next definition.
	
\begin{definition}[bounded geometry]\label{bg}
A metric space $X$ is said to have \textit{bounded geometry} if for every $r\geq 0$ there exists an $N_r \in \N$ 
such that for all $x\in X$, the ball of radius $r$ about $x$ has at most $N_r$ elements.
\end{definition}

Finally in this section, we recall a general fact about derivations.
	
\begin{definition}[spatial derivation]
Let $A\subseteq \BB(\HH)$ be a concrete $C^*$-algebra.  A derivation $\delta$ of $A$ is \textit{spatial} if there is a bounded operator $d \in\BB(\HH)$ such that $\delta(a) = [a,d]$.
\end{definition}
	
The following is due to Kadison \cite[Theorem 4]{10.2307/1970433}.
	
\begin{theorem}\label{spatial the}
Let $A\subseteq \BB(\HH)$ be a concrete $C^*$-algebra.  Then every derivation on $A$ is spatial. \qed
\end{theorem}
	
Note that a uniform Roe algebra $C^*_u(X)$ always contains the compact operators on $\ell^2(X)$.  For a concrete $C^*$-algebra $A\subseteq \BB(\HH)$ containing the compact operators $\KK(\HH)$, there are simpler proofs of Theorem \ref{spatial the} available: see for example \cite[Corollary 3.4 and Remark on page 284]{Chernoff:1973aa}.

\section{Averaging operators over amenable groups}  

In this section, we summarize some facts we need about averaging operators on a Hilbert space over an amenable group.  Most of this material seems likely to be well-known; however, we could not find convenient references for the facts we wanted, so provide most details here.

Let $G$ be a discrete (possibly uncountable) group.  If $A$ is a complex Banach space, we let $\ell^\infty(G,A)$ denote the Banach space of bounded functions from $G$ to $A$ equipped with the supremum norm; in the case $A=\C$, we just write $\ell^\infty(G)$.  We also equip $\ell^\infty(G,A)$ with the right-action of $G$ defined for $a\in \ell^\infty(G,A)$ and $h,g\in G$ by
$$
(ag)(h):=a(hg).
$$
If $Z$ is any set, a function $\phi:\ell^\infty(G,A)\to Z$ is \emph{invariant} if $\phi(ag)=\phi(a)$ for all $a\in \ell^\infty(G,A)$ and $g\in G$.  Recall that $G$ is amenable if there exists an \emph{invariant mean} on $\ell^\infty(G)$, i.e.\ an invariant function $\phi:\ell^\infty(G)\to \C$ that is also a state. 

Fix now an invariant mean on $\ell^\infty(G)$, which we denote\footnote{The integral notation is meant to be suggestive, but we do not need to, and will not, assign any specific meaning to the `measure' $\mu$.} by
$$
a\mapsto \int_{G}a(g)\dd\mu(g).
$$
Let now $B$ be a complex Banach space with dual $B^*$.  We may upgrade an invariant mean on $\ell^\infty(G)$ to an invariant contractive linear map $\ell^\infty(G,B^*)\to B^*$ in the following way. Let $b\in B$, $g\in G$, and $a\in \ell^\infty(G,B^*)$, and write $\langle b,a(g)\rangle$ for the pairing between $b$ and $a(g)$.  Then the map 
$$
G\to \C,\quad g\mapsto \langle b,a(g)\rangle
$$
is bounded, and so we may apply the invariant mean to get a complex number 
$$
\int_{G}\langle b,a(g)\rangle\dd\mu(g).
$$
It is not difficult to check that the map 
$$
B\to \C, \quad b\mapsto \int_{G}\langle b,a(g)\rangle\dd\mu(g)
$$
is a bounded linear functional on $B$.  We write $\int_{G}a(g)\dd\mu(g)$ for this bounded linear functional.

The following lemma is straightforward: we leave the details to the reader.  

\begin{lemma}\label{avg lem}
With notation as above, the map
$$
\ell^\infty(G,B^*)\to B^*,\quad a\mapsto \int_{G}a(g)\dd\mu(g)
$$
is uniquely determined by the condition 
\begin{equation}\label{avg}
\Big \langle b,\int_{G}a(g)\dd\mu(g)\Big\rangle=\int_{G}\langle b,a(g)\rangle\dd\mu(g)
\end{equation}
for $b\in B$ and $a\in \ell^\infty(G,B^*)$.  It is contractive, linear, invariant, and acts as the identity on constant functions. \qed
\end{lemma}

We will apply this machinery in the case that $B=\mathcal{L}^1(\ell^2(X))$ is the trace class operators on $\ell^2(X)$.  In this case, the dual $B^*$ canonically identifies with $\BB(\ell^2(X))$: indeed, if $\text{Tr}$ is the canonical trace $\mathcal{L}^1(\ell^2(X))$, $b\in \mathcal{L}^1(\ell^2(X))$, and $a\in \BB(\ell^2(X))$, then the pairing inducing this duality isomorphism is defined by
\begin{equation}\label{pair form}
\langle b,a\rangle:=\text{Tr}(ba).
\end{equation} 

We will need some basic lemmas.  The first can be deduced very quickly from the theory of conditional expectations (see for example \cite[Lemma 1.5.10]{Brown:2008qy}); we instead give a slightly longer naive proof.

\begin{lemma}\label{ce prop}
With notation as above, for any $a\in \ell^\infty(G,\BB(\ell^2(X))$ and $c\in \BB(\ell^2(X))$, we have that 
$$
c\int_{G}a(g)\dd\mu(g)=\int_{G}ca(g)\dd\mu(g) \quad \text{and}\quad \int_{G}a(g)\dd\mu(g)c=\int_{G}a(g)c\dd\mu(g)
$$
\end{lemma}

\begin{proof}
Using lines \eqref{avg} and \eqref{pair form}, for any $b\in \mathcal{L}^1(\ell^2(X))$, we have 
\begin{align*}
\Big\langle b, c\int_{G}a(g)\dd\mu(g)\Big\rangle & =\text{Tr}\Big(bc\int_{G}a(g)\dd\mu(g)\Big)=\Big\langle bc, \int_{G}a(g)\dd\mu(g)\Big\rangle \\ & =\int_{G}\langle bc,a(g)\rangle\dd\mu(g)=\int_{G} \text{Tr}(bca(g))\dd\mu(g) \\ & =\int_{G}\langle b,ca(g)\rangle\dd\mu(g)  =\Big\langle b, \int_{G}ca(g)\dd\mu(g)\Big\rangle.
\end{align*}
As $b\in \mathcal{L}^1(\ell^2(X))$ was arbitrary, this implies that $c\int_{G}a(g)\dd\mu(g)=\int_{G}ca(g)\dd\mu(g)$.  The other case is similar, using also the trace identity $\text{Tr}(cd)=\text{Tr}(dc)$, which is valid whenever either $c$ or $d$ is trace class.
\end{proof}

The next lemma says that our averaging process behaves well with respect to propagation.  Again, we proceed naively; the key point of the lemma is that the collection of operators in $\BB(\ell^2(X))$ that have propagation at most $r$ is weak-$*$ closed for the weak-$*$ topology inherited from the pairing with $\mathcal{L}^1(\ell^2(X))$.

\begin{lemma}\label{avg prop lem}
With notation as above, if $r\geq 0$ and $a\in \ell^\infty(G,\BB(\ell^2(X)))$ is such that the propagation of each $a(g)$ is at most $r$, then the propagation of $\int_{G}a(g)\dd\mu(g)$ is also at most $r$.
\end{lemma}

\begin{proof}
Let $e_{xy}\in \mathcal{L}^1(\ell^2(X))$ be the standard matrix unit.  Then one computes using line \eqref{pair form} above that for any $a\in \BB(\ell^2(X))$,
\begin{equation}\label{munit}
\langle e_{yx},a\rangle=\text{Tr}(e_{yx}a)=a_{xy}.
\end{equation}
Using lines \eqref{avg} and \eqref{munit}, we see that
$$
\Big\langle e_{yx},\int_{G}a(g)\dd\mu(g)\Big\rangle= \int_{G}\langle e_{yx}, a(g)\rangle\dd\mu(g)=\int_{G}a(g)_{xy}\dd\mu(g),
$$
where the last expression means the image of the function 
$$
G\to \C, \quad g\mapsto a(g)_{xy}
$$
under the invariant mean.  If $d(x,y)>r$, we have that $a(g)_{xy}=0$ for all $g\in G$, and therefore that $\int_{G}a(g)_{xy}\dd\mu(g)=0$.  Hence by the above computation, 
$$
d(x,y)>r \quad \text{implies} \quad \Big\langle e_{yx},\int_{G}a(g)\dd\mu(g)\Big\rangle=0.
$$
Using line \eqref{munit}, this says that $\int_{G}a(g)\dd\mu(g)$ has propagation at most $r$, so we are done.
\end{proof}

\begin{lemma}\label{com lem}
With notation as above, say that there is a unitary representation $g\mapsto u_g$ of $G$ on $\ell^2(X)$.  For any fixed $d\in \BB(\ell^2(X))$, define $a\in \ell^\infty(G,\BB(\ell^2(X)))$ by $a(g):=u_g^*du_g$.  Then $\int_{G}a(g)\dd\mu(g)$ is in the commutant of the set $\{u_g\mid g\in G\}$.
\end{lemma}

\begin{proof}
Let $h\in G$.  Then by Lemma \ref{ce prop},
$$
u_h\int_{G}u_g^*du_g\dd\mu(g)=\int_{G}u_hu_g^*du_g\dd\mu(g)=\int_{G}u_{gh^{-1}}^*du_g\dd\mu(g).
$$
Making the `change of variables' $k=gh^{-1}$ and using right-invariance of the map $a\mapsto \int_{G}a(g)\dd\mu(g)$, this equals
$$
\int_{G}(u_{k})^*du_{kh}\dd\mu(k)=\int_{G}u_{k}^*du_{k}u_h\dd\mu(k).
$$
Using Lemma \ref{ce prop} again we get $\int_{G}u_{k}^*du_{k}u_h\dd\mu(k)=\int_{G}u_{k}^*du_{k}\dd\mu(k)u_h$, so are done.
\end{proof}

\section{Proof of the main result}

In this section, we start by summarizing a fact we need from the recent work of Braga-Farah.  We then prove Theorem \ref{main}.	

To state the result due to Braga and Farah \cite[Lemma 4.9]{Braga:2018dz}, let $\D:=\{z\in \C:|z|\leq 1\}$ denote the closed unit disk, and for a set $I$, let $\D^I$ denote the usual product space of functions $I\to \D$.  We write elements of $\D^I$ as tuples $\overline{\lambda}=(\lambda_i)_{i\in I}$. 

\begin{lemma}\label{rigid}
Let $(X,d)$ be a metric space with bounded geometry, and let $I$ be a countable set.  Suppose that $(a_i)_{i\in I}$ is a family of finite rank operators in $\cu{X}$ such that for every $\overline{\lambda}\in \D^I$ the series $\sum_{i\in I} \lambda_i a_i$ 
converges strongly to an operator $a_{\overline{\lambda}} \in \cu{X}$.
Then for every $\epsilon >0$ there exists $r>0$ such that 
$a_{\overline{\lambda}}$ can be $\epsilon$-$r$-approximated for all $\overline{\lambda}\in \D^I$. \qed
\end{lemma}
	
We are now ready for the proof of Theorem \ref{main}.

\begin{proof}[Proof of Theorem \ref{main}]
Let $\delta:\cu{X}\to \cu{X}$ be a derivation.  Theorem \ref{spatial the} implies that $\delta$ is spatially implemented, so there is $d\in \BB(\ell^2(X))$ such that $\delta(a)=[a,d]$ for all $a\in C^*_u(X)$.  We will show that $d$ is in $C^*_u(X)$.

Let $\mathcal{U}$ be the unitary group of $\ell^{\infty}(X)$, equipped with the discrete topology.  As $\mathcal{U}$ is abelian, it is amenable (see for example \cite[Theorem G.2.1]{Bekka:2000kx}), and so we may fix a right-invariant mean on $\ell^\infty(\mathcal{U})$.  As in Lemma \ref{avg lem} above, this allows us to build a right-invariant, contractive, linear map 
\begin{equation}\label{u avg}
\ell^\infty(\mathcal{U},\BB(\ell^2(X)))\to \BB(\ell^2(X)),\quad a\mapsto \int_{\mathcal{U}} a(u)\dd\mu(u).
\end{equation}
We apply this to the bounded function 
$$
\mathcal{U}\to \BB(\ell^2(X)),\quad u\mapsto u^*du
$$
to get a bounded operator
$$
d':=\int_{\mathcal{U}}u^*du\dd\mu(u)\in \BB(\ell^2(X)).
$$
Using Lemma \ref{com lem} applied to the identity representation of $\mathcal{U}$, $d'$ is in the commutant of $\mathcal{U}$.  As $\mathcal{U}$ spans $\ell^\infty(X)$, and as $\ell^\infty(X)$ is maximal abelian in $\BB(\ell^2(X))$, this implies that $d'$ is in $\ell^\infty(X)$.  To show that $d$ is in $C^*_u(X)$, it therefore suffices to show that $h:=d-d'$ is in $C^*_u(X)$.  

Continuing, let $p_x\in \BB(\ell^2(X))$ be the rank one projection onto the span of the Dirac mass at $x$.  For an element $f$ of the unit ball of $\ell^\infty(X)$ (considered as a multiplication operator on $\ell^2(X)$), write $f$ as a strongly convergent sum 
$$
f=\sum_{x\in X} f(x) p_x.
$$
Then using strong continuity of subtraction, and separate strong continuity of multiplication on bounded sets,
$$
[f,d]=\Big[\sum_{x\in X} f(x)p_x,d\Big]=\sum_{x\in X} f(x)[p_x,d].
$$
On the other hand, by assumption that $\delta$ is a derivation on $C^*_u(X)$, $[f,d]$ is in $C^*_u(X)$ for all $f\in \ell^\infty(X)$. It follows that if we set $I=X$, and if for each $x\in X$ we set $a_x:=[p_x,d]$, then the collection $(a_x)_{x\in X}$ satisfies the assumptions of Lemma \ref{rigid}.  Hence, for every $\epsilon>0$ there exists $r>0$ such that for every $f$ in the unit ball of $\ell^\infty(X)$, the operator $[f,d]$ can be $\epsilon$-$r$ approximated.  In particular, using that any $u\in\mathcal{U}$ has propagation zero and norm one, for any $\epsilon>0$ there exists $r>0$ such that $d-u^*du=u^*[u,d]$ can be $\epsilon$-$r$ approximated.

For each $u\in \mathcal{U}$, we can therefore choose $a(u)$ of propagation at most $r$ such that $b(u):=d-u^*du-a(u)$ has norm at most $\epsilon$.  Note that the functions $a:u\mapsto a(u)$ and $b:u\mapsto b(u)$ are in $\ell^\infty(\mathcal{U},\BB(\ell^2(X)))$.  Hence we may consider their images under the map in line \eqref{u avg}.  Using that the map in line \eqref{u avg} is linear and acts as the identity on constant functions (see Lemma \ref{avg lem}), we see that
\begin{align}\label{h sum}
\int_{\mathcal{U}} a(u)\dd\mu(u)+\int_{\mathcal{U}} b(u)\dd\mu(u) & =\int_{\mathcal{U}} d-u^*du\dd\mu(u)=d-\int_{\mathcal{U}} u^*du\dd\mu(u) \nonumber \\ & =d-d'  =h.
\end{align}
On the other hand, $\int_{\mathcal{U}} a(u)\dd\mu(u)$ has propagation at most $r$ by Lemma \ref{avg prop lem}, and $\int_{\mathcal{U}} b(u)\dd\mu(u)$ has norm at most $\epsilon$ as the map in line \eqref{u avg} is contractive (see Lemma \ref{avg lem}).  In particular, line \eqref{h sum} writes $h$ as a sum of an element of $C^*_u(X)$, and an element of norm at most $\epsilon$.  As $\epsilon$ was arbitrary, $h$ is in $C^*_u(X)$, and we are done.  
\end{proof}

\bibliographystyle{abbrv}
\bibliography{biblio,Generalbib}
\end{document}